\documentclass{amsart}
\usepackage{MnSymbol}
\usepackage{scalerel}

\usepackage{hyperref}
\hypersetup{
    colorlinks,
    citecolor=black,
    filecolor=black,
    linkcolor=black,
    urlcolor=black
}
\usepackage{comment}

\DeclareFontFamily{OT1}{pzc}{}
\DeclareFontShape{OT1}{pzc}{m}{it}{<-> s * [1.100] pzcmi7t}{}
\DeclareMathAlphabet{\mathpzc}{OT1}{pzc}{m}{it}

    \usepackage{etoolbox}
    \patchcmd{\section}{\scshape}{\large\bfseries}{}{}
    \makeatletter
    \renewcommand{\@secnumfont}{\bfseries}
    \makeatother

\usepackage{tikz-cd}
\tikzcdset{arrow style=math font}

\tikzcdset{scale cd/.style={every label/.append style={scale=#1},
    cells={nodes={scale=#1}}}}
\usepackage{tikz}
\usetikzlibrary{arrows}
\usepackage{caption}
\usepackage{float}

\usepackage[maxbibnames=99]{biblatex}
\usepackage{booktabs}
\bibliography{references}

\numberwithin{equation}{section}
\newtheorem{theorem}{Theorem}[section]
\newtheorem*{theorem*}{Theorem}
\newtheorem{corollary}[theorem]{Corollary}
\newtheorem{lemma}[theorem]{Lemma}
\newtheorem{proposition}[theorem]{Proposition}
\theoremstyle{definition}

\newtheorem*{question*}{Question}

\newtheorem{remark}[theorem]{Remark}
\newtheorem{example}[theorem]{Example}

\def\Ker{\mathrm{Ker}}
\def\Im{\mathrm{Im}}

\def\KK{\mathbb{K}}
\def\RR{\mathbb{R}}
\def\ZZ{\mathbb{Z}}
\def\SH{\mathrm{SH}}
\def\RC{\mathrm{RC}}
\def\MH{\mathrm{MH}}

\let\oldtocsection=\tocsection 
\let\oldtocsubsection=\tocsubsection 
\renewcommand{\tocsection}[2]{\hspace{0mm}\oldtocsection{#1}{#2}}
\renewcommand{\tocsubsection}[2]{\hspace{1em}\oldtocsubsection{#1}{#2}}

\title[Nested homotopy models of finite metric spaces]{Nested homotopy models of finite metric spaces and their spectral homology} 

\author{Sergei O. Ivanov} 
\address{
Beijing Institute of Mathematical Sciences and Applications}
\email{ivanov.s.o.1986@gmail.com, ivanov.s.o.1986@bimsa.cn}

\begin{document}

\begin{abstract} For a real $r\geq 0,$ we consider the notion of $r$-homotopy equivalence in the category quasimetric spaces, which includes metric spaces and directed graphs. We show that for a finite quasimetric space $X$ there is a unique (up to isometry) $r$-homotopy equivalent quasimetric space of the minimal possible cardinality. It is called the $r$-minimal model of $X$.  We use this to construct a decomposition of the magnitude-path spectral sequence of a digraph into a direct sum of spectral sequences with certain properties. 

We also construct an $r$-homotopy invariant ${\rm SH}^r_{n,I}(X)$ of a quasimetric space $X,$ called spectral homology, that generalizes many other invariants: the pages of the magnitude-path spectral sequence, including path homology, magnitude homology, blurred magnitude homology and reachability homology.  
\end{abstract}

\maketitle

\section{Introduction}
\subsection{Background}

For over a decade, two theories have been actively developed. The first theory is the theory of magnitude of metric spaces and their magnitude homology. This theory was developed by Leinster, Meckes, Shulman, Hepworth, Willerton, Yoshinaga, Asao and others 
\cite{leinster2013magnitude,meckes2013positive,leinster2019magnitude,hepworth2015categorifying,leinster2021magnitude,kaneta2021magnitude,asao2021geometric}. It explores an invariant of a metric space called magnitude. On one hand, the magnitude is a real number, which is a continuous analogue of the cardinality of a finite set. On the other hand, it exhibits some properties similar to those of the Euler characteristic. The concept of magnitude is also related to the notion of diversity \cite{leinster2016maximizing}, \cite{leinster2021entropy}. Subsequently it was discovered that it is possible to introduce a homology theory for metric spaces, so that for finite metric spaces the Betti numbers of this theory are related to the  magnitude in the similar way as ordinary Betti numbers are related to the Euler characteristic. This homology theory is called magnitude homology. 

The second theory is the GLMY-theory of digraphs. This theory was developed by Grigor'yan, Lin, Muranov, Yau and others  \cite{grigor2012homologies,grigor2014homotopy, grigor2017homologies, grigor2020path, grigor2014graphs, grigor2018path,grigor2018path2,grigor2022advances,ivanov2022simplicial}. It implements a homotopical approach to the theory of digraphs. Its central concept is the path homology of a digraph, also known as GLMY-homology. The path homology behaves in many ways similar to the homology of spaces. For example, there is an analogue of the K\"unneth formula; there is a notion of suspension of a digraph that shifts the path homology; there is a notion of the fundamental group of a digraph, whose abelianization equals the first path homology group, and so forth. The theory also has the notion of homotopy of morphisms in the category of digraphs, and path homology is homotopy invariant. Moreover, there is a theory of covering digraphs, analogous to the theory of covering spaces in algebraic topology \cite{di2023path}.

Asao later showed that for the case of digraphs, these two theories could be unified into a single theory \cite{asao2023magnitude},\cite{asao2023magnitude2}. He showed that each digraph $G$ defines a spectral sequence $E(G)$, whose first page is the magnitude homology and the second page contains the path homology as the zeroth row. This spectral sequence is called magnitude-path spectral sequence. He also defines the concept of $r$-homotopy on the category of digraphs and shows that $E^{r+1}(G)$ is an $r$-homotopy invariant. For $r=1$, this concept reduces to homotopy as defined in the GLMY-theory. Thus, this result generalizes the fact that the path homology is a homotopy invariant.

For a digraph $G$, one can define the distance between vertices $d_G(g,g')$ as the length of the shortest path from $g$ to $g'$. If there is no path from $g$ to $g'$, we assume $d_G(g,g')=+\infty.$ Note that this distance is not necessarily symmetric and can be equal to infinity. However, it satisfies the triangle inequality and $d_G(g,g')=0$ if and only if $g=g'.$ Such generalizations of metric spaces are referred to as quasimetric spaces. We will work in the category of quasimetric spaces where morphisms are short maps.

The aim of our work is to study the concept of $r$-homotopy, introduced by Asao, for the case of quasimetric spaces and any real $r\geq 0.$ We study this concept from two perspectives. First, we study finite quasimetric spaces that have the least cardinality among quasimetric spaces in a given $r$-homotopy equivalence class. Secondly, we introduce a new invariant of the $r$-homotopy equivalence for any quasimetric spaces, and study this invariant. Both parts are related to the magnitude-path spectral sequence: in the first part we obtain its decomposition into a direct sum of some smaller spectral sequences; in the second part we generalize its pages to the case of any quasimetric spaces. 

\subsection{Nested \texorpdfstring{$r$}{}-minimal models} In this theory, after the $r$-homotopy equivalence of short maps is defined,
 $r$-homotopy equivalences and $r$-deformation retracts are defined similarly to how it is defined in the classical algebraic topology. We say that a finite quasimetric space $M$ is $r$-minimal, if it has no proper $r$-deformation retracts. A quasimetric space $M$ is called the $r$-minimal model of a quasimetric space $X,$ if $M$ is $r$-minimal and it is $r$-homotopy equivalent to $X.$ We prove the following theorem.
\begin{theorem*} Let $r\geq 0$ be a real number. Then the following holds.
\begin{enumerate}
    \item Any finite quasimetric space has an $r$-minimal model, which is unique up to isometry. 
    \item Two finite quasimetric spaces are $r$-homotopy equivalent if and only if their $r$-minimal models are isometric. 
    \item Any finite quasimetric space has an $r$-deformation retract, which is its $r$-minimal model. 
\end{enumerate}
\end{theorem*}

The $r$-minimal model of a space $X$ is denoted by $M_r(X).$ Note that we do not claim that $X\mapsto M_r(X)$ is a functor, and the inclusion $M_r(X) \hookrightarrow X$ and the retraction $X\to M_r(X)$ are not uniquely defined. It is easy to see that if $r'>r,$ then $M_{r'}(M_r(X))\cong M_{r'}(X).$ Using this we obtain that there is a finite number of real numbers $0<r_1<\dots < r_n$ such that $M_{r_i}(X)\not\cong M_r(X)$ for $r<r_i.$ We call these points homotopy jumping points and set
\begin{equation}
{\rm hj}(X)=\{r_1,\dots,r_n\}.
\end{equation}
Thus for any finite quasi-metric space $X$ we obtain a nested sequence of retracts 
\begin{equation}\label{eq:nested}
X \supset M_{r_1}(X) \supset M_{r_2}(X) \supset \dots \supset M_{r_n}(X).
\end{equation} 
Unfortunately, the map $X\mapsto {\rm hj}(X)$ from the set of finite subsets of $\RR^n$ to the set of finite subsets of $\RR$ is not continuous with respect to the Hausdorff distance (Example \ref{ex:discontinuity}).

If $X=G$ is a digraph, the sequence \eqref{eq:nested} is a sequence of convex subdigraphs, such that $M_{r_{i+1}}(G)$ is a retract of $M_{r_i}(G).$ Using this we obtain a decomposition of the magnitude-path spectral sequence $E(G)$ into a direct sum of spectral sequences
\begin{equation}\label{eq:decomp_spectral}
E(G) = \bigoplus_{k=1}^n E(G)_{[k]} \oplus E(G)_{[\infty]}
\end{equation}
such that $E^{r_k+1}(G)_{[k]}=0$ and for any $1\leq j\leq n$ we have 
\begin{equation}
E(M_{r_j}(G)) \cong \bigoplus_{k=j+1}^n E(G)_{[k]} \oplus E(G)_{[\infty]}. 
\end{equation}
However, we do not claim that this decomposition is functorial by $G$. 

We say that a finite digraph $G$ is homotopy stable, if $M_r(G)=G$ for any $r\geq 0.$ Any finite digraph $G$ defines a homotopy stable digraph by the formula $M_{\sf stab}(G)=M_r(G),$ where $r$ is greater then all homotopy jumping points of $G$. The decomposition \eqref{eq:decomp_spectral} shows that the magnitude-path spectral sequence $E(G)$ contains a summand corresponding to the homotopy stable digraph $E(M_{\sf stab}(G))$ such that $E^r(M_{\sf stab}(G))=E^r(G)$ for all big enough $r.$ This indicates that homotopy stable digraphs may be an important subject for further study.

\subsection{Spectral homology}

The magnitude-path spectral sequence is defined  for the discrete case of digraphs and its first page is the magnitude homology. However the magnitude homology is defined for all quasimetric spaces. Spectral homology is an attempt to define the magnitude-path spectral sequence for any quasimetric space so that its ``first page'' is the magnitude homology. In some sense the spectral homology is a ``continuous'' generalization of the magnitude-path spectral sequence. In order to define the spectral homology of a quasimetric space, we need to define the spectral homology of an $\RR$-filtered chain complex so that the spectral homology is a generalization of pages of the spectral sequence associated with a $\ZZ$-filtered chain complex. In other words, the spectral homology of an $\RR$-filtered complex is an analogue of the spectral sequence of a $\ZZ$-filtered chain complex.

For an $\RR$-filtered chain complex $C$ over a ring $\KK,$ a real number $r\geq 0,$ a non-empty interval $I\subseteq \RR$ and a natural number $n$ we define a $\KK$-module $\SH^r_{n,I}(C),$ called spectral homology, or $r$-spectral homology, if we want to specify $r.$ The spectral homology generalizes the spectral sequence associated with a $\ZZ$-filtered chain complex in the following sense. If $C$ is a $\ZZ$-filtered chain complex, then its $\ZZ$-filtration can be extended to an $\RR$-filtration by the formula $F_sC = F_{\lfloor s\rfloor}C$ for $s\in \RR.$ In this case for a natural $r\geq 1$ the pages of the associated spectral sequence can be presented as the spectral homology $E^r_{n,\ell-n}(C)=\SH^{r-1}_{n,\{\ell\}}(C).$ 

Let us give the definition of $\SH^r_{n,I}(C).$ A left infinite interval in $\RR$ is an interval of the form $(-\infty,a],(-\infty,a), \RR, \emptyset.$ For a left infinite interval $L$  we set $F_L=\bigcup_{s\in L} F_sC.$ If we present a non-empty interval $I$ as the difference of left infinite intervals $I=R\setminus L,$  then 
\begin{equation}
\SH^r_{n,I}(C) = \Im( H_n( F_{R}/F_{L-r} ) \longrightarrow H_n( F_{R+r}/F_L ) ).
\end{equation}
Moreover, there is another more technical equivalent definition of the spectral homology that shows that for $s\geq r\geq 0$ the module $\SH^{s}_{n,I}(C)$ is a subquotient of $\SH^{r}_{n,I}(C)$
(see details in Subsection \ref{subsection:spectral_homology}).

The magnitude-path spectral sequence of a digraph $G$ is defined using a $\ZZ$-filtered chain complex $\RC(G),$ which is called reachability chain complex (see \cite{hepworth2023reachability}). This complex is defined similarly for the case of quasimetric space $X.$ However, in this case we obtain an $\RR$-filtration on the chain complex $\RC(X).$ Then the spectral homology of $X$ is defined as the spectral homology of the $\RR$-filtered chain complex
\begin{equation}
\SH^r_{n,I}(X) = \SH^r_{n,I}(\RC(X)).
\end{equation}
We show that $\SH^r_{n,I}(X)$ is an $r$-homotopy invariant. 

The spectral homology is natural by the interval $I$ in some sense. In particular, there is a homomorphism $\SH^r_{n,[0,\ell]}(X) \to \SH^r_{n,[0,\ell']}(X)$ for $\ell<\ell'.$ This allows us to define a persistent module ${\rm PSH}^r_n(X)$ such that ${\rm PSH}^r_n(X)(\ell)=\SH^r_{n,[0,\ell]}(X).$ 

We prove that some other types of homology can be presented as spectral homology.

\begin{theorem*} Let $X$ be a quasimetric space and $G$ be a digraph. Then the following holds.
\begin{enumerate}
\item The magnitude homology can be presented as $0$-spectral homology
$
{\rm MH}_{n,\ell}(X)\cong \SH^0_{n,\{\ell\}}(X).$
\item The persistent module of  blurred magnitude homology (see \cite{govc2021persistent}) can be presented as the persistent $0$-spectral homology ${\rm BMH}_{n}(X)\cong {\rm PSH}^0_{n}(X).$
\item The reachability homology (see \cite{hepworth2023reachability}) can be presented as $r$-spectral homology  
${\rm RH}_n(X)\cong \SH^r_{n,\RR}(X)$ for any $r\geq 0.$
\item For any integer $r\geq 1$ the $r$-th page of the magnitude-path spectral sequence of $G$ can be presented as $(r-1)$-spectral homology 
$E^r_{n,\ell-n}(G)\cong \SH^{r-1}_{n,\{\ell\}}(G).$
\item In particular, the path homology can be described as
$
{\rm PH}_n(G)=\SH^1_{n,\{n\}}(G).$
\end{enumerate}
\end{theorem*}

In the end of the paper we give a description of some low-dimensional spectral homology. In particular, for $\ell>r\geq 0,$ we give a description of $\SH^r_{1,\{\ell\}}(X)$ in terms of adjacent pairs, similar to the description of $\MH_{1,\ell}(X)$ given in \cite[Cor. 4.5]{leinster2021magnitude}. 
We also work out the following continuous  example. Let $X \subseteq S^1$ be an arc in the circle of radius $1$ (with the Euclidean distance) that does not lie in one semicircle. 
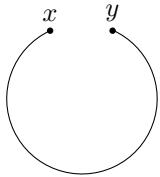
\[
\begin{tikzpicture}
\draw (0,0) arc (115:425:1);
\filldraw  (0,0) circle (1pt) node[anchor=south]{$x$};
\filldraw  (0.83,0) circle (1pt)
node[anchor=south]{$y$};
\end{tikzpicture}
\]
If we assume that $d(x,y)<\sqrt{2}/2$ and $0<r<\ell,$ then 
\begin{equation}
\SH^r_{1,\{\ell\}}(X) \cong 
\begin{cases}
 \ZZ^2,& \ell=d(x,y),   \\
 0, & \ell\ne d(x,y).
\end{cases}
\end{equation} 
So the first spectral homology detects the distance $d(x,y)$ for the ``singular pair'' $(x,y),$ and shows that the arc is not $r$-contractible for $r<d(x,y).$ Moreover, in this case $\SH^r_{1,\{\ell\}}(X)$ has finite rank, while $\MH_{1,\ell}(X)$ has infinite rank.

\subsection{Acknowledgments} The author is grateful to Lev Mukoseev for fruitful discussions and Example \ref{example:Lev}.

\tableofcontents

\section{\texorpdfstring{$r$}{}-minimal models of quasimetric spaces}

\subsection{Quasimetric spaces} 
A \emph{quasimetric space} is a set $X$ equipped by a map $d:X\times X\to [0,+\infty]$ such that the triangle inequality holds $d(x,z)\leq d(x,y)+d(y,z)$ for any $x,y,z\in X$ and $d(x,y)=0$ if and only if $x=y.$ Note that we allow the distance to be equal to infinity, and it is not necessary symmetric, but the distance between distinct points is positive.  Two sources of examples of quasimetric spaces are metric spaces and directed graphs with the shortest path distance. 

A map between two quasimetric spaces $\varphi:X\to Y$ is called short (or 1-Lipschitz, or distance-decreasing), if
\begin{equation}
d(\varphi(x),\varphi(x'))\leq d(x,x')    
\end{equation}
for any $x,x'\in X.$ The set of all short maps from $X$ to $Y$ is denoted by ${\rm SMap}(X,Y).$ The category of quasimetric spaces with short maps as morphisms is denoted by ${\sf QMet}.$ Note that isomorphisms in this category are isometries (distance preserving bijections). 

For two short maps $\varphi,\psi: X\to Y$ we set 
\begin{equation}
d(\varphi,\psi)  = \underset{x}{\rm sup} \ d(\varphi(x),\psi(x)).
\end{equation}
It is easy to see that this defines a structure of quasimetric space on ${\rm SMap}(X,Y).$ Note that for any short maps $\beta:Y\to Z$ and $\alpha:W\to X$ we have 
\begin{equation}\label{eq:dist_comp}
d(\beta \varphi \alpha,\beta \psi \alpha) \leq d(\varphi,\psi).
\end{equation}

\subsection{\texorpdfstring{$r$}{}-homotopy category}

Let $r\geq 0$ be a real number. We say that two points $x,x'$ of a quasimetric space $X$  are $r$-homotopic, denoted by $x\sim_r x',$ if there is a sequence $x=x_0,\dots,x_n=x'$ of points of $X$ such that for each $0\leq i\leq n-1$ we have either $d(x_i,x_{i+1})\leq r$ or $d(x_{i+1},x_i)\leq r.$ In other words, $\sim_r$ is the minimal equivalence relation such that $d(x,x')\leq r$ implies $x\sim_r x'.$ We say that two short maps $\varphi,\psi:X\to Y$ are $r$-homotopic if they are $r$-homotopic as points in ${\rm SMap}(X,Y).$ 

\begin{lemma}\label{lemma:congruence} If $\alpha:W\to X,$  $\varphi,\psi:X\to Y$ and $\beta:Y\to Z$ are short maps between quasimetric spaces and $\varphi\sim_r \psi,$ then $\beta\varphi \alpha\sim_r \beta\psi\alpha.$  In other words, $\sim_r$ is a congruence relation on the category ${\sf QMet}.$
\end{lemma}
\begin{proof}
Follows from \eqref{eq:dist_comp}.
\end{proof}

Since $\sim_r$ is a congruence relation on ${\sf QMet},$ we can define the $r$-homotopy category of quasimetric spaces as the quotient-category ${\sf QMet}/\sim_r$. We say that a short map is an $r$-homotopy equivalence, if its image in ${\sf QMet}/\sim_r$ is an isomorphism, and two quasimetric spaces are $r$-homotopy equivalent, if they are isomorphic in ${\sf QMet}/\sim_r.$

For a quasimetric space $X$ we denote by ${\rm End}(X)$ the monoid of endomorphisms of $X$ in ${\sf QMet}.$ We will also denote by ${\rm End}_r(X)$ the submonoid consisting of short maps $\varphi :X\to X$ such that $\varphi \sim_r 1_X.$ 
\begin{equation}
{\rm End}_r(X) = \Ker({\rm End}(X) \to {\rm End}_{{\sf QMet}/\sim_r}(X)).
\end{equation}

\subsection{\texorpdfstring{$r$}{}-deformation retracts}

If $X$ is a  quasimetric space and $A\subseteq X$ is a subset, $A$ will be treated as a  quasimetric space with the induced metric. We will denote by $\iota=\iota_{A,X}:A\to X$ the canonical embedding. A \emph{retract} of a quasimetric space $X$ is a subset $A\subseteq X$ such that there exists a short map $\rho:X\to A$, called retraction, such that  $\rho\iota=1_A$. We say that the retract is \emph{proper}, if $A\ne X.$ An \emph{$r$-deformation retract} is a retract $A\subseteq X$ with a retraction $\rho$ such that $\iota \rho \sim_r 1_X.$

\begin{proposition}\label{prop:endomorphism}
Let $X$ be a finite quasimetric space and  $\varphi\in {\rm End}_r(X)$. Then there exists a natural $n\geq 1$ such that $\varphi^n$ is idempotent, and for any such $n$ the subset $\varphi^n(X)\subseteq X$ is an $r$-deformation retract of $X.$
\end{proposition}
\begin{proof} Since $X$ is finite, there exists such $n$ that $\varphi^n$ is an idempotent.  
Lemma \ref{lemma:congruence} implies that $1\sim_r \varphi \sim_r \varphi^2 \sim_r \dots \sim_r \varphi^n.$ Therefore $\varphi^n\sim_r 1_X.$ Set $A=\varphi^n(X).$ Then we can present $\varphi^n$ as the composition $\varphi^n= \iota \rho,$ where $ \iota :A\to X$ is the canonical embedding and $\rho:X\to A$ is some surjection. Since the quasimetric on $A$ is induced from $X,$ we obtain that $\rho$ is a short map. Note that $\iota \rho\sim_r 1_X.$ Since $ \varphi^n$ is idempotent, we have $\iota \rho\iota \rho=\iota \rho.$ Using that $\iota$ is injective, and $\rho$ is surjective, we obtain $\rho\iota =1_A.$
\end{proof}

\subsection{\texorpdfstring{$r$}{}-minimal spaces} We say that a finite quasimetric space $M$ is  $r$-minimal, if it has no proper $r$-deformation retracts.

\begin{proposition}\label{prop:minimal-equi}
Let $X$ be a finite quasimetric space. Then the following statements are equivalent.
\begin{enumerate}
\item $X$ is $r$-minimal;
\item ${\rm End}_r(X)\subseteq {\rm Aut}(X)$;
\item any $\varphi\in {\rm End}(X)$ such that either $d(1_X,\varphi)\leq r$ or $d(\varphi,1_X)\leq r$ is an automorphism. 
\end{enumerate}  
\end{proposition}
\begin{proof}
$(1) \Rightarrow (2).$ Take $\varphi\in {\rm End}_r(X).$ Proposition \ref{prop:endomorphism} implies that there exists $n\geq 1$ such that $\varphi^n$ is an idempotent and $\varphi^n(X)\subseteq X$ is an $r$-deformation retract. Since $X$ is $r$-minimal, we obtain $\varphi^n(X)=X.$ Using that $X$ is finite, we obtain that $\varphi^n$ is a bijection and an isometry. 
Combining this with the fact that $\varphi$ an  idempotent, we obtain $\varphi^n=1_X.$ Therefore $\varphi$ is an automorphism.

$(2)\Rightarrow (3).$ Obvious.

$(3)\Rightarrow (1).$  Consider an $r$-deformation retract $A\subseteq X$ with a retraction $\rho:X\to A.$ We need to prove that $A=X.$ Since $\iota \rho\sim_r 1_X,$ there exists a sequence $1_X= \varphi_0,\dots, \varphi_n=\iota \rho$ such that for any $0\leq i\leq n-1$ either $d(\varphi_i,\varphi_{i+1})\leq r$ or $d(\varphi_{i+1},\varphi_i)\leq r.$ Prove by induction that $\varphi_i$ is an automorphism. The base of induction $i=0$ is obvious. Assume that $\varphi_i$ is an automorphism. If $d(\varphi_i,\varphi_{i+1})\leq r,$ then $d(1_X,\varphi_{i}^{-1}\varphi_{i+1})\leq r,$ and hence $\varphi_{i}^{-1}\varphi_{i+1}$ is an automorphism. If $d(\varphi_{i+1},\varphi_i)\leq r,$ then $d(\varphi_{i+1}\varphi^{-1}_i,1_X)\leq r,$ and hence, $\varphi_{i+1}\varphi^{-1}_i$ is an automorphism. It follows that $\varphi_{i+1}$ is an automorphism. Since $\iota\rho$ is an automorphism, $A=X.$ 
\end{proof}

\begin{proposition}\label{prop:equivalence_of_minimal_models}
Let $M$ and $M'$ be $r$-minimal finite quasimetric spaces. If $\varphi:M\to M'$ is an $r$-homotopy equivalence, then $\varphi$ is an isometry. 
\end{proposition}
\begin{proof}
Since $\varphi$ is an $r$-homotopy equivalence, there is a short map $\psi:M'\to M$ such that $\psi\varphi\sim_r 1_M$
and $\varphi \psi\sim_r 1_{M'}.$ By Proposition \ref{prop:minimal-equi} we obtain that $\psi\varphi$ and $\varphi \psi$ are automorphisms. Therefore $\varphi$ is an isomorphism. 
\end{proof}

\subsection{\texorpdfstring{$r$}{}-minimal model of a space} 

We say that $M$ is an $r$-minimal model of a finite quasimetric space $X,$ if $M$ is $r$-minimal, and $X$ is $r$-homotopy equivalent to $M.$ 

\begin{theorem}\label{th:minimal_models} Let $r\geq 0$ be a real number. Then the following holds.
\begin{enumerate}
    \item Any finite quasimetric space has an $r$-minimal model, which is unique up to isometry. 
    \item Two finite quasimetric spaces are $r$-homotopy equivalent if and only if their $r$-minimal models are isometric. 
    \item Any finite quasimetric space has an $r$-deformation retract, which is its $r$-minimal model. 
\end{enumerate}
\end{theorem}
\begin{proof}For a finite quasimetric space $X$ we take any $r$-deformation retract $M\subseteq X$ of the minimal possible cardinality. Denote the retraction by $\rho:X\to M.$  We claim that $M$ is $r$-minimal. Indeed, if $A\subseteq M$ is an $r$-deformation retract with a retraction $\rho':M\to A$ then the equations $\rho\iota_{M,X}=1_M$ and $\rho'\iota_{A,M}=1_A$ imply $\rho'\rho\iota_{A,X}=1_A,$ and relations $\iota_{M,X}\rho\sim_r 1_X$ and $\iota_{A,M}\rho'\sim_r 1_M$ imply $\iota_{A,X} \rho'\rho \sim_r 1_X.$ Hence $A$ is an $r$-deformation retract of $X,$ which is subset of $M,$ and using the minimality of the cardinality of $M,$ we obtain $A=M.$ 

Proposition \ref{prop:equivalence_of_minimal_models} implies that the $r$-minimal model is defined uniquely up to isometry. It also implies that $r$-homotopy equivalent quasimetric spaces have isomorphic $r$-minimal homotopy models. 
\end{proof}

The $r$-minimal model of a quasimetric space $X$ is denoted by $M_r(X).$  Note that for $r'>r$ we have 
\begin{equation}
M_{r'}(M_r(X)) \cong M_{r'}(X).
\end{equation}

\begin{example}
Let $X$ be a finite metric space (with symmetric finite distance). For $r>{\rm diam}(X),$ the metric space $M_r(X)$ is a one-point space. Here we see that the embedding $\iota:M_r(X)\to X$ is not uniquely defined, because we can take any map from the one-point space to $X$ as the inclusion. 
\end{example}

\begin{example} Consider the following graph $X$
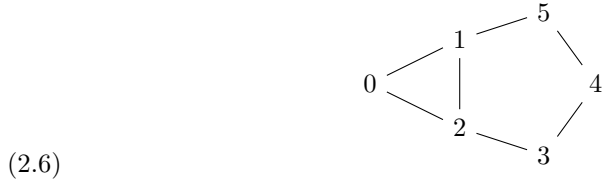
\begin{equation}
\begin{tikzpicture}
\draw[line width=1pt]
(-2,0) node (b)
{$0$}
({cos(0*360/5)},{sin(0*360/5)})  node (a1) {$4$}
({cos(1*360/5)},{sin(1*360/5)})  node (a2) {$5$}
({cos(2*360/5)},{sin(2*360/5)})  node (a3) {$1$}
({cos(3*360/5)},{sin(3*360/5)})  node (a4) {$2$}
({cos(4*360/5)},{sin(4*360/5)})  node (a5) {$3$}
;
\path[-]
(a5) edge node {} (a1) 
(a1) edge node {} (a2)
(a2) edge node {} (a3)
(a3) edge node {} (a4)
(a4) edge node {} (a5)
;
\path[-]
(b) edge node {} (a3)
(b) edge node {} (a4)
;
\end{tikzpicture}
\end{equation}
as a metric space with the shortest path distance. Take $r=1.$ Then the $1$-minimal model $M_1(X)$ is the pentagon with the vertex set $\{1,2,3,4,5\}.$ There are two choices of retractions for this $1$-deformation retract $\rho,\rho':X\to M_1(X):$  $\rho(0)=1$ and $\rho'(0)=2.$ In particular, this example shows that the retraction $\rho:X\to M_r(X)$ can't be defined by a universal property in ${\sf QMet}$. 
\end{example}

\begin{example}\label{example:Lev}
The author is grateful to Lev Mukoseev for the following non-trivial but clarifying example.
Consider the following digraph $X$ with four vertices together with a subdigraph $A\subseteq X$ with three vertices 
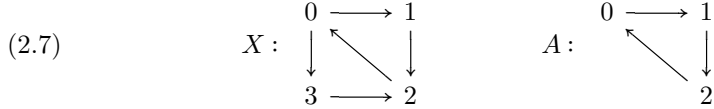
\begin{equation}
X:\ 
\begin{tikzcd}
0 \ar[r] \ar[d] &1 \ar[d]\\
3\ar[r] &2 \ar[lu]
\end{tikzcd}
\hspace{1.5cm} 
 A:\ 
\begin{tikzcd}
0 \ar[r]  &1 \ar[d]\\
 &2 \ar[lu]
\end{tikzcd}
\end{equation}
It is easy to see that $A$ is a retract of $X,$ because there is a retraction $\rho:X\to A$ such that $\rho(3)=1.$ 
We claim that $A$ is a $1$-deformation retract of $X$ and $A=M_1(X).$ In particular 
\begin{equation}
X \sim_1 A.
\end{equation}
In order to prove it consider an endomorphism $\varphi:X\to X$ defined by 
\begin{equation}
0\mapsto 2, \hspace{5mm} 1\mapsto 0, \hspace{5mm}  2\mapsto 1, \hspace{5mm} 3\mapsto 0.
\end{equation}
It can be depicted as follows.
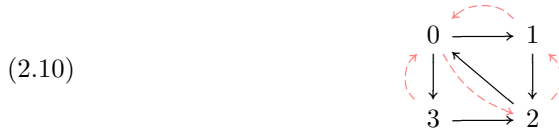
\begin{equation}
\begin{tikzcd}
0 \ar[r] \ar[d] \ar[rd,bend right = 7mm,red!50,dashed] &1 \ar[d] \ar[l,bend right = 15mm,red!50,dashed] \\
3\ar[r] \ar[u,bend left = 15mm,red!50,dashed] &2 \ar[lu] \ar[u,bend right = 15mm,red!50,dashed]
\end{tikzcd}   
\end{equation}
Then $d(\varphi,1_X)=1$ and $\iota \rho =\varphi^3.$ It follows that $\iota\rho= \varphi^3 \sim_1 \varphi^2 \sim_1 \varphi \sim_1 1_X.$
\end{example}

\begin{remark} We say that a finite metric space $X$ is in general position if for any two pairs of distinct points $x\neq x'$ and $y\neq y'$ such that $\{x,x'\}\ne \{y,y'\}$ the distances are not equal $d(x,x')\neq d(y,y').$ Then for any finite metric space in general position, if $M_r(X)$ consists of more then one point, then the embedding $\iota : M_r(X)\to X$ is uniquely defined.
\end{remark}

\subsection{Homotopy jumping points} 
 For a finite quasimetric space $X$ we consider a function from non-negative real numbers to the interval $c_X:[0,+\infty) \to [0,1]$ defined as the ratio of cardinalities  
\begin{equation}
c_X(r)= \frac{|M_r(X)|}{|X|}.
\end{equation}

\begin{lemma}\label{lemma:decreasing_function}
The function $c_X$ is decreasing,  lower semi-continuous and it has a finite number of values. 
\end{lemma}
\begin{proof}
If $r'>r,$ then $M_{r'}(M_r(X))\cong M_{r'}(X),$ and hence $c_X(r')\leq c_X(r).$ Now assume that $r_n\to r$ and prove that $c_X(r)\leq \underset{n\to +\infty}{\rm liminf} \: c_X(r_n).$ Since $c_X$ is decreasing, we can assume that $r_n>r$ is an decreasing sequence. Since $X$ is finite, for $n>\!>1$ we have that for any two points $x,x'$ either $d(x,x')\leq r$ or $d(x,x')>r_n.$ Therefore, for $n>\!>1$ the inequality $d(x,x')\leq r_n$ is equivalent to the inequality $d(x,x')\leq r.$ Hence for $n>\!>1$ we have  $M_r(X)=M_{r_n}(X)$ and $c_X(r_n)=c_X(r).$
\end{proof}

Lemma \ref{lemma:decreasing_function} implies that $c_X$ has a finite number of jumping points  $0<r_1<\dots<r_n$ such that 
$c_X(r_i)<c_X(r)$
for any $r<r_i.$ We will call them homotopy jumping points of $X$ and set
\begin{equation}
{\rm hj}(X)=\{r_1,\dots,r_n\}.
\end{equation}
If we also set $r_0=0$ and $r_{n+1}=+\infty,$ then  
\begin{equation}
M_r(X)\cong M_{r_i}(X), \hspace{1cm} r\in [r_i,r_{i+1}).
\end{equation}

\begin{remark}
Let us call a quasimetric space $X$ homotopy stable, if ${\rm hj}(X)=\emptyset.$ In other words, $X$ is homotopy stable, if $M_r(X)=X$ for any $r\geq 0.$ If we set ${\rm End}_{\rm fin}(X)=\bigcup_r {\rm End}_r(X),$ then Proposition \ref{prop:minimal-equi} says that $X$ is homotopy stable if and only if ${\rm End}_{\rm fin}(X)\subseteq {\rm Aut}(X).$ For example, the quasimetric spaces defined by the digraphs 
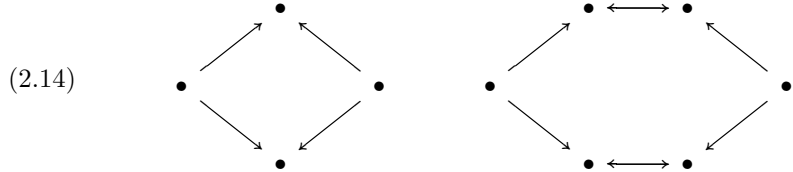
\begin{equation}
\begin{tikzcd}
& \bullet  &  \\
\bullet\ar[ru] \ar[rd] &  & \bullet \ar[lu] \ar[ld] \\
& \bullet  &    
\end{tikzcd} \hspace{1cm}
\begin{tikzcd}
& \bullet \ar[r] & \bullet \ar[l] & \\
\bullet\ar[ru] \ar[rd] & & & \bullet \ar[lu] \ar[ld] \\
& \bullet \ar[r] & \bullet \ar[l] &    
\end{tikzcd}
\end{equation}
are homotopy stable.
\end{remark}

Any finite quasimetric space $X$ with the set of homotopy jumping points 
$\{r_1<\dots<r_n\}$ defines a homotopy stable finite quasimetric space given by
\begin{equation}
M_{\rm stab}(X) := M_{r_n}(X),
\end{equation}
if $n>0,$ and $M_{\rm stab}(X)=X,$ if $n=0.$ For a finite quasimetric space $X$ with a point $x_0\in X$ such that $d(x,x_0)<+\infty$ for any $x\in X,$ the metric space $M_{\rm stab}(X)$ is always a one-point space. 

\begin{example}[Discontinuity of ${\rm hj}$ with respect to the Hausdorff distance]\label{ex:discontinuity}
For any $0\leq \varepsilon<1/2$ we consider a subset of $\mathbb{R}^2$ of six points 
\begin{equation}
X_\varepsilon=(\{-2,2\}\times \{0,1\})\cup\{(-\varepsilon,0),(\varepsilon,1)\},
\end{equation}
which can be depicted as follows.
\begin{equation}
\begin{tikzpicture}
\draw[line width=1pt]
(-2,1) node (m21)
{$\bullet$}
(-2,0) node (m20)
{$\bullet$}
(0,1) node (01)
{$\circ$}
(0.3,1) node (e1)
{$\bullet$}
(0,0) node (00)
{$\circ$}
(-0.3,0) node (me0)
{$\bullet$}
(2,1) node (21)
{$\bullet$}
(2,0) node (20)
{$\bullet$}
;
\end{tikzpicture}
\end{equation}
Note that $\underset{\varepsilon\to 0+}\lim X_\varepsilon = X_0$ with respect to the Hausdorff distance. For $\varepsilon=0$ we have $M_1(X_0)\cong \{-2,0,2\}\times\{0\},$ because there is a retraction $\rho:X_0\to \{-2,0,2\}\times\{0\}$ defined by $\rho((a,b))=(a,0)$ such that $\iota \rho\sim_1 1_X.$ It is also easy to see that $M_2(X_0)$ is a one-point set. Using this we get 
\begin{equation}
{\rm hj}(X_0)=\{1,2\}.
\end{equation}
However, for $\varepsilon>0$ the direct verification using brute force 
shows that for any non-identical short map 
$\varphi:X_\varepsilon\to X_\varepsilon$ we have $d(\varphi,1_X)=d(1_X,\varphi)\geq \sqrt{1+(2-\varepsilon)^2}.$ A non-identical short map 
$\varphi:X_\varepsilon\to X_\varepsilon$ 
with the minimal possible $d(1,\varphi)$ is the map that sends $(-2,0),(-2,1),(-\varepsilon,0)$ to $(-\varepsilon,0)$ and sends $(2,0),(2,1),(\varepsilon,1)$ to $(\varepsilon,1).$ Using this we obtain 
\begin{equation}
{\rm hj}(X_\varepsilon)=\{\sqrt{1+(2-\varepsilon)^2}\}.
\end{equation}
Since $\underset{\varepsilon\to 0+}\lim \{\sqrt{1+(2-\varepsilon)^2}\}=\{\sqrt{5}\}\ne \{1,2\}$, the map $X\mapsto {\rm hj}(X)$ is not continuous with respect to the Hausdorff distance. 
\end{example}

\section{Magnitude-path spectral sequence of a digraph}

\subsection{\texorpdfstring{$r$}{}-minimal models of digraphs} We define a digraph as a pair $G=(G_0,G_1),$ where $G_1\subseteq G_0\times G_0$ such that $(g,g)\in G_1$ for any $g\in G_0.$ Elements of $G_0$ are called vertices, elements of $G_1$ are called arrows, arrows of the form $(g,g)$ are called degenerate arrows. A morphism of digraphs $\varphi:G\to H$ is a map $\varphi:G_0\to H_0$ such that $(g,g')\in G_1$ implies $(\varphi(g),\varphi(g'))\in H_1.$ The category of digraphs is denoted by ${\sf Digr}.$ A path in a digraph $G$ is a tuple $(g_0,\dots,g_n)$ such that $(g_i,g_{i+1})\in G_1.$ Note that any morphism sends a path to a path. 

Any digraph $G$ defines a quasimetric space, that we denote by ${\rm Q}(G),$ whose set of points is $G_0$ and the distance from $g$ to $g'$ is defined as the infimum of the lengths of paths from $g$ to $g'.$ Therefore we get a functor
\begin{equation}
{\rm Q}:{\sf Digr} \longrightarrow {\sf QMet}.
\end{equation}
It is easy to check that the functor is full and faithful. So we can identify ${\sf Digr}$ with a full subcategory of ${\sf QMet}.$ So we can define $r$-homotopic maps for digrphs.

A subdigraph $H$ of a digraph $G$ is called convex, if the shortest path distance on $H$ coincides with the induced distance: $d_H(h,h')=d_G(h,h')$ for any vertices $h,h'$ of $H.$ A subdigraph $H\subseteq G$ is called a retract, if there is a morphism of digraphs $\rho:G\to H$ such that $\rho \iota=1_H,$ where $\iota:H\hookrightarrow G$ is the canonical embedding. 

\begin{lemma}
A retract of a digraph is a convex subdigraph. 
\end{lemma}
\begin{proof} Let $H$ be a retract of $G.$
It is obvious that $d_H(h,h')\geq d_G(h,h')$ for any $h,h'\in H_0.$ On the other hand, for any path $(g_0,\dots,g_n)$ from $h$ to $h'$ in $G$  we have a path $(\rho(g_0),\dots,\rho(g_n))$ from $h$ to $h'$ in $H.$ Hence $d_H(h,h')\leq d_G(h,h').$
\end{proof}

\begin{lemma}\label{lemma:retract_digrph}
Let $G$ be a digraph and $A$ be a retract of ${\rm Q}(G).$ Then $A={\rm Q}(H)$ for a uniquely defined retract $H\subseteq G.$ 
\end{lemma}
\begin{proof} Consider a subdigraph $H\subseteq G$ such that $H_0=A$ and $H_1=\{(h,h')\mid d_G(h,h')\leq 1\}.$ Denote by $\rho:{\rm Q}(G)\to A$ a retraction. If $(g,g')\in G_1,$ then $d_G(g,g')\leq 1,$ and hence $(\rho(g),\rho(g'))\in H_1.$ So $\rho$ defines a morphism of digraphs $\rho':G\to H.$ Therefore $H$ is a retract of $G.$ In particular $H$ is a convex subdigraph. It follows that the distance between points in $A$ coincides with the distance in ${\rm Q}(H).$ 
\end{proof}

Lemma \ref{lemma:retract_digrph} implies that for a finite digraph $G$  and a natural $n$ there is a digraph $M_n(G)$ such that ${\rm Q}(M_n(G)) = M_n({\rm Q}(G)),$ that can be included into $G$ as an $n$-deformation retract, but the embedding is not unique. Therefore for a finite digraph we obtain a nested sequence of convex sudigraphs 
\begin{equation}
G \supseteq M_1(G) \supseteq M_2(G) \supseteq \dots,
\end{equation}
where the inclusions are not uniquely defined, but the digraphs $M_n(G)$ are defined uniquely up to isomorphism.

\subsection{Reminder on the magnitude-path spectral sequence}
Here we remind the construction and the main properties of the magnitude-path spectral sequence of a digraph that can be found in 
\cite{asao2023magnitude},  \cite{asao2023magnitude2}, \cite{hepworth2023reachability},\cite{di2023path}.

Let $\KK$ be a commutative ring. For a set $S$ we will denote by $\KK\cdot S$ the free $\KK$-module freely generated by $S.$
For a digraph $G$ we consider a filtered chain complex ${\rm RC}_\bullet(G)$ over $\KK,$ called reachability complex, whose $n$-th component is freely generated by tuples of vertices  $(g_0,\dots,g_n)$ without consecutive repetitions $g_i\neq g_{i+1}$ such that $d(g_i,g_{i+1})<+\infty.$ 
\begin{equation}
{\rm RC}_n(G) =\KK\cdot \{(g_0,\dots,g_n)\mid d(g_i,g_{i+1})<+\infty, g_i\ne g_{i+1}\}.
\end{equation}
Tuples $(g_0,\dots,g_n)$ with a consecutive repetition $g_i=g_{i+1}$ will be identified with zero in the chain complex. Then the differential is defined by the formula
\begin{equation}
\partial(g_0,\dots,g_n) = \sum_{i=0}^n(-1)^i (g_0,\dots,\hat g_i,\dots,g_n).
\end{equation}
If we set 
$L(g_0,\dots,g_n)=\sum_i d(g_i,g_{i+1}),$
the filtration on the chain complex is defined so that for an integer $\ell$ the component $F_\ell{\rm RC}_n(G)$ is generated by tuples $(g_0,\dots,g_n)$ such that $L(g_0,\dots,g_n)\leq \ell.$ 
\begin{equation}
F_0{\rm RC}(G)\subseteq F_1{\rm RC}(G)\subseteq \dots \subseteq {\rm RC}(G).    
\end{equation}
The spectral sequence associated with this filtered chain complex is called magnitude-path spectral sequence of $G$ and will be denoted by $E(G).$ Its first page consists of the magnitude homology
\begin{equation}
E^1_{\ell,n-\ell}(G) = {\rm MH}_{n,\ell}(G),
\end{equation}
the zero row of the second page is the path homology
\begin{equation}
E^2_{n,0}(G) = {\rm PH}_n(G).
\end{equation}
Since the filtration of $\RC(G)$ is exhaustive and bounded below, the spectral sequence converges to the homology of the chain complex \cite[Th.5.5.1(2)]{weibel1995introduction} (here by convergence we mean  ``unbounded'' convergence in the sense of unbounded spectral sequences \cite[\S 5.2.11]{weibel1995introduction})
\begin{equation}
E(G)\Rightarrow H_*(\RC(G)).
\end{equation}
The homology of this chain complex is known as reachability homology and denoted by ${\rm RH}_*(G).$ The reachability homology is equal to the homology of the reachability preorder ${\rm Pre}(G),$ which is defined so that its elements are points of $G$ and $g\leq g'$ if and only if $d(g,g')< \infty$ 
\begin{equation}
{\rm RH}_n(G) = H_n(\RC(G)) \cong H_n({\rm Pre}(G)).
\end{equation}

The $r+1$-st page of the magnitude-path spectral sequence is an $r$-homotopy invariant in the following sense. If $\varphi\sim_r\psi:G\to H$ two $r$-homotopy invariant morphisms, then they induce equal morphisms on the $r+1$-st pages
\begin{equation}
E^{r+1}(\varphi)=E^{r+1}(\psi) : E^{r+1}(G) \longrightarrow E^{r+1}(H)
\end{equation}
(see \cite[Th. 4.37]{asao2023magnitude2}).

\subsection{Decomposition of the magnitude-path spectral sequence}

\begin{theorem}\label{theorem:decomposition_magnitude-path} Let $G$ be a finite digraph and let
${\rm hj}(G)=\{r_1<\dots<r_n\}$
be the set of homotopy jumping points of $G$. 
Then the magnitude-path spectral sequence $E(G)$ can be decomposed into a direct sum of spectral sequences
\begin{equation}\label{eq:decomp}
E(G) = \bigoplus_{k=1}^n E(G)_{[k]} \oplus E(G)_{[\infty]}
\end{equation}
such that $E^{1+r_k}(G)_{[k]}=0$ and for any $1\leq j\leq n$ we have 
\begin{equation}\label{eq:decomp2}
E(M_{r_j}(G)) \cong \bigoplus_{k=j+1}^n E(G)_{[k]} \oplus E(G)_{[\infty]}. 
\end{equation}
In particular, $E(M_{\rm stab}(G))=E(G)_{[\infty]}.$
\end{theorem}
\begin{proof} It is easy to see that, 
if $E$ is a spectral sequence and $e:E\to E$ is an idempotent morphism of spectral sequences, then $E$ can be decomposed into a sum of spectral sequences $E=\Ker(e)\oplus \Im(e).$ In other words, the category of spectral sequences is Karoubian. Therefore, if $E'$ is a retract of a spectral sequence $E,$ then $E$ is a direct sum of spectral sequences $E\cong E'\oplus E''.$

If $n=0,$ then there is nothing to prove. Assume $n\geq 1.$ Chose a sequence of nested minimal models
\begin{equation}
G=M_{r_0}(G)\supset M_{r_1}(G) \supset M_{r_2}(G) \supset \dots \supset M_{r_n}(G)   
\end{equation}
with retractions $\rho_j:M_{r_j}(G)\to M_{r_{j+1}}(G)$ for $0\leq j\leq n-1.$ Since $M_{r_{j+1}}(G)$ is a retract in $M_{r_j}(G),$ we obtain that  $E(M_{r_{j+1}})$ is a retract of $E(M_{r_j}).$ Hence $E(M_{r_j})\cong E(G)_{[j+1]} \oplus  E(M_{r_{j+1}}).$ Then we obtain a decomposition \eqref{eq:decomp} with the property \eqref{eq:decomp2}. Since the embedding $M_{r_{j+1}}(G) \hookrightarrow M_{r_j}(G)$ is an $r_{j+1}$-homotopy equivalence, we obtain that $E^{r_{j+1}+1}(M_{r_{j+1}}(G))\cong E^{r_{j+1}+1}(M_{r_j}(G)).$ Therefore $E^{r_{j+1}+1}(G)_{[j+1]}=0.$
\end{proof}
\begin{remark}
We do not claim that the decomposition in Theorem \ref{theorem:decomposition_magnitude-path} is natural. 
\end{remark}

\section{Spectral homology}

\subsection{Reminder on spectral sequence of a \texorpdfstring{$\ZZ$}{}-filtered chain complex} 

In this subsection we remind an  explicit construction of pages of the spectral sequence associated with a filtered chain complex that can be found in \cite[Ch.XV,\S 1]{cartan1999homological}. We will use the following lemma from the book of Cartan-Eilenberg. 
\begin{lemma}[{\cite[Ch.XV,Lemma 1.1]{cartan1999homological}}]\label{lemma:CE} For any commutative diagram of $\KK$-modules of the following
form, in which the row is exact
\begin{equation}
\begin{tikzcd}
&B \ar[d,"\varphi"] \ar[dr,"\psi"] & \\
A'\ar[r,"\varphi'"']\ar[ru] &A \ar[r,"\eta"'] &A'
\end{tikzcd}
\end{equation}
there is an inclusion $\Im \: \varphi'\subseteq \Im\: \varphi$ and 
the map $\eta$ induces an isomorphism
\begin{equation}
\Im\: \psi \cong \Im \: \varphi/\Im \: \varphi'.
\end{equation}
\end{lemma}

Let $\KK$ be a commutative ring and $C$ be a $\ZZ$-filtered chain complex over $\KK$ with an increasing filtration
\begin{equation}
\dots \subseteq F_{-1}C \subseteq F_0C\subseteq F_1C \subseteq \dots \subseteq C.
\end{equation}
For the sake of simplicity we set 
\begin{equation}
F_\ell = F_\ell C.
\end{equation}
Consider the associated spectral sequence $E.$ Here we will not discuss the convergence of the spectral sequence, so we do not need any assumptions on the filtration. The first page of this spectral sequence is described as
\begin{equation}
E^1_{\ell,n-\ell} = H_n(F_\ell/F_{\ell-1}).
\end{equation}
Let us describe all pages as subquotients of the first page. For any natural $r\geq 1$ we consider the following  short exact sequences induced by the embeddings
\begin{equation}\label{eq:ses_of_comp1}
0 \longrightarrow F_{\ell-1}/F_{\ell-r} \longrightarrow F_\ell/F_{\ell-r} \longrightarrow F_\ell/F_{\ell-1} \longrightarrow 0
\end{equation}
and 
\begin{equation}\label{eq:ses_of_comp2}
0 \longrightarrow F_\ell/F_{\ell-1} \longrightarrow F_{\ell+r-1}/F_{\ell-1} \longrightarrow F_{\ell+r-1}/F_\ell \longrightarrow 0,
\end{equation}
and the associated long exact sequences of homology. Further we set
\begin{equation}
\begin{split}
Z^r_{\ell,n-\ell} &= \Im( 
H_n(F_\ell/F_{\ell-r})\longrightarrow  H_n(F_\ell/F_{\ell-1}) )\\
&=\Ker( 
H_n(F_\ell/F_{\ell-1}) \longrightarrow H_{n-1}(F_{\ell-1}/F_{\ell-r})).
\end{split}
\end{equation}
and
\begin{equation}
\begin{split}
B^r_{\ell,n-\ell} &= \Ker(H_n(F_\ell/F_{\ell-1}) \longrightarrow  H_n(F_{\ell+r-1}/F_{\ell-1}))\\
& = \Im(H_{n+1}(F_{\ell+r-1}/F_\ell) \longrightarrow  H_n(F_\ell/F_{\ell-1})).
\end{split}
\end{equation}
Then the $r$-th page of the spectral sequence can be described as
\begin{equation}
E^r_{\ell,n-\ell} = \frac{Z^r_{\ell,n-\ell}}{B^r_{\ell,n-\ell}}.
\end{equation}
Using the morphism of short exact sequences,
\begin{equation}
\begin{tikzcd}
0 \ar[r] & F_\ell/F_{\ell-r} \ar[r]\ar[d,twoheadrightarrow] & F_{\ell+r-1}/F_{\ell-r}\ar[r]\ar[d,twoheadrightarrow] & F_{\ell+r-1}/F_\ell\ar[r]\ar[d,equal] & 0\\
0 \ar[r] & F_\ell/F_{\ell-1} \ar[r] & F_{\ell+r-1}/F_{\ell-1}\ar[r] & F_{\ell+r-1}/F_\ell\ar[r] & 0
\end{tikzcd}
\end{equation}
we obtain a commutative diagram, in which the row is exact.
\begin{equation}
\begin{tikzcd}
& H_n(F_\ell/F_{\ell-r}) \ar[d] \ar[dr] & \\
H_{n+1}(F_{\ell+r-1}/F_\ell) \ar[r]\ar[ru] & H_n(F_\ell/F_{\ell-1})\ar[r] & H_n(F_{\ell+r-1}/F_{\ell-1}) 
\end{tikzcd}
\end{equation}
Applying Lemma \ref{lemma:CE} to this diagram,  
we obtain another description of the $E^r$ (see \cite[Ch.XV,\S 1,(8)]{cartan1999homological})
\begin{equation}
E^r_{\ell,n-\ell} \cong  \Im( H_n(F_\ell /F_{\ell-r} ) \longrightarrow H_n( F_{\ell+r-1}/F_{\ell - 1} )).
\end{equation}

\subsection{Spectral homology of an
\texorpdfstring{$\RR$}{}-filtered chain complex}\label{subsection:spectral_homology}
A subset of the set of real numbers $I\subseteq \RR$ is called an interval if it is convex. We say that a interval $L$ is left infinite, if for any $i\in I$ and $r<i$ we have $r\in L.$ There are four types of left infinite intervals: $(-\infty,a),(-\infty,a], \emptyset$ and $\RR.$ Any non-empty interval $I\subseteq \RR$ can be uniquely presented as the difference of left infinite intervals $I=R_I\setminus L_I,$ where 
\begin{equation}
R_I=\{r\in \RR\mid \exists i\in I\ r\leq i\}, \hspace{1cm} L_I=\{r\in \RR \mid \forall i\in I\ r<i \}.
\end{equation} 
For example $R_{\{\ell\}}=(-\infty,\ell]$ and $L_{\{\ell\}}=(-\infty,\ell).$ We define an order on the set of non-empty intervals such that 
$I \preccurlyeq J$
if and only
if $R_I\subseteq R_J$ and $L_I\subseteq L_J.$ For example, $[0,2] \preccurlyeq [1,3].$

For an interval $I=R\setminus L$ and a non-negative real number $r\geq 0$ we also set 
$I+r = \{i+r\mid i\in I\}$
and
\begin{equation}
I_{-r} = R\setminus (L-r), \hspace{1cm} I^{+r} = (R+r)\setminus L.    
\end{equation}
Note that 
\begin{equation}
 I_{-r}  \preccurlyeq I \preccurlyeq I^{+r}.
\end{equation}

Assume that $C$ is an $\RR$-filtered chain complex over a ring $\KK$. So for any $r\in \RR$ we have a chain subcomplex $F_rC\subseteq C$ such that  $r<r'$ implies $F_rC\subseteq F_{r'}C.$ When $C$ is fixed we will omit it in notations $F_r=F_rC.$ For a left infinite interval $L$ we set 
\begin{equation}
F_L = \bigcup_{r\in L} F_r
\end{equation}
Further, for a non-empty interval $I=R\setminus L$ we set
\begin{equation}
C_I = F_R/F_L.
\end{equation}
If $I \preccurlyeq J,$ we have a morphism 
\begin{equation}
C_I \longrightarrow C_{J}
\end{equation}
induced by inclusions $F_{R_I}C\subseteq F_{R_J}C$ and $F_{L_I}C\subseteq F_{L_J}C$. Moreover, if $L_I=L_J,$ this morphism is a monomorphism, and, if $R_I=R_J,$ it is an epimorphism. If we have three distinct left infinite  intervals $R\supset S \supset L,$ then these morphisms define a short exact sequence of chain complexes
\begin{equation}\label{eq:ses}
0 \longrightarrow C_{S\setminus L} \longrightarrow C_{R\setminus L} \longrightarrow C_{R\setminus S} \longrightarrow 0.
\end{equation}

The $0$-spectral homology of $C$ is defined as follows
\begin{equation}\label{eq:SH^0}
\SH^{0}_{n,I}(C) = H_n(C_I).
\end{equation}
Further for any $r\geq 0$ we will define $r$-spectral homology  $\SH^r_{n,I}(C)$ as a subquotient of $\SH^0_{n,I}(C)$

For any $r\geq 0$ we consider short exact sequences of chain complexes of the form \eqref{eq:ses} similar to \eqref{eq:ses_of_comp1}, \eqref{eq:ses_of_comp2}
\begin{equation}\label{eq:ses_of_comp1'}
0
\longrightarrow
C_{I_{-r} \setminus I}
\longrightarrow 
C_{I_{-r}} 
\longrightarrow C_I \longrightarrow 0, 
\end{equation}
\begin{equation} \label{eq:ses_of_comp2'}
0 \longrightarrow C_I \longrightarrow C_{I^{+r}} \longrightarrow C_{I^{+r}\setminus I} \longrightarrow 0
\end{equation}
and the associated long exact sequences of homology. Then we define $r$-spectral cycles and $r$-spectral boundaries as 
\begin{equation}
\begin{split}
{\rm SZ}^r_{n,I}(C)&= \Im( H_n(C_{I_{-r}}) 
\longrightarrow H_n(C_I))\\
&=\Ker( H_n(C_I) \longrightarrow H_{n-1}(C_{I_{-r}\setminus I}))
\end{split}
\end{equation}
and 
\begin{equation}
\begin{split}
{\rm SB}^r_{n,I}(C)& = \Ker( H_n(C_I) \longrightarrow H_n( C_{I^{+r}})) \\
& = \Im(H_{n+1}(C_{I^{+r}\setminus I}) \longrightarrow H_n(C_I)).
\end{split}
\end{equation}

\begin{proposition}\label{prop:second_descr_SH}
There is an inclusion 
\begin{equation}
{\rm SB}^r_{n,I}(C) \subseteq {\rm SZ}^r_{n,I}(C).
\end{equation}
Moreover, if we set
\begin{equation}
\SH^r_{n,I}(C) = {\rm SZ}^r_{n,I}(C)/{\rm SB}^r_{n,I}(C),
\end{equation}
then the map $C_I\to C_{I^{+r}}$ induces an isomorphism 
\begin{equation}
\SH^r_{n,I}(C) \cong \Im( H_n(C_{I_{-r}}) \longrightarrow H_n(C_{I^{+r}})).
\end{equation}
\end{proposition}
\begin{proof}
Using that $I^{+r}\setminus I=(I^{+r} \cup I_{-r}) \setminus I_{-r},$ we obtain that there is a morphism of short exact sequences of the form \eqref{eq:ses}
\begin{equation}
\begin{tikzcd}
0 \ar[r] & C_{I_{-r}} \ar[r]\ar[d,twoheadrightarrow] &
C_{I^{+r}\cup I_{-r} }\ar[r]\ar[d,twoheadrightarrow] & 
C_{I^{+r}\setminus I} \ar[r]\ar[d,equal] & 0\\
0 \ar[r] & C_I \ar[r] & C_{I^{+r}} \ar[r] & C_{I^{+r}\setminus I} \ar[r] & 0
\end{tikzcd}
\end{equation}
Therefore there is an commutative diagram, in which the row is exact.
\begin{equation}
\begin{tikzcd}
& H_n(C_{I_{-r}}) \ar[d] \ar[dr] & \\
H_{n+1}(C_{I^{+r}\setminus I}) \ar[r]\ar[ru] & H_n(C_I)\ar[r] & H_n(C_{I^{+r}}) 
\end{tikzcd}
\end{equation}
Then the assertion follows from Lemma \ref{lemma:CE}.
\end{proof}

The $\KK$-module $\SH^r_{n,I}(C)$ defined in Proposition \ref{prop:second_descr_SH} is called \emph{$r$-spectral homology} of the $\RR$-filtered chain complex $C.$ A morphism of $\RR$-filtered chain complexes $f:C\to C'$ is a chain map such that $f(F_\ell C)\subseteq F_\ell C'.$ It is easy to see that the construction is natural and defines a functor
\begin{equation}
\SH^r_{n,I} : \RR\text{-}{\sf FiltCh(\KK)} \longrightarrow {\sf Mod}(\KK).
\end{equation}

\begin{corollary}\label{cor:SH-good-formula}
There are isomorphisms 
\begin{equation}
\begin{split}
{\rm SZ}^r_{n,I}(C)&= \Im(
\SH^0_{n,I_{-r}}(C) \longrightarrow \SH^0_{n,I}(C)),
\\
{\rm SB}^r_{n,I}(C)& = \Ker( \SH^0_{n,I}(C) \longrightarrow \SH^0_{n,I^{+r}}(C)),
\\
 \SH^r_{n,I}(C) &\cong \Im( \SH^0_{n,I_{-r}}(C) \longrightarrow \SH^0_{n,I^{+r}}(C)).
\end{split}
\end{equation}

\end{corollary}

\begin{proposition}\label{prop:spectral=spectral}
Let $C$ be a $\ZZ$-filtered chain complex. We extend the $\ZZ$-filtration to an $\RR$-filtration on $C$ by the formula 
\begin{equation}
F_sC = F_{\lfloor s\rfloor}C.
\end{equation}
Then for a natural number $r\geq 1$ we have
\begin{equation}
E^{r}_{n,\ell-n}(C) = 
\SH^{r-1}_{n,(\ell-1,\ell]}(C)=\SH^{r-1}_{n,\{\ell\}}(C).  
\end{equation}
\end{proposition}
\begin{proof}
For $I=(\ell-1,\ell]$ we obtain $R=(-\infty,\ell-1]$ and $L=(-\infty,\ell].$ Therefore, $F_R=F_\ell,$ $F_L=F_{\ell-1},$ $F_{L-(r-1)}=F_{\ell-r},$ $F_{R+(r-1)}=F_{\ell+r-1}.$ Then the short exact sequences \eqref{eq:ses_of_comp1}, \eqref{eq:ses_of_comp2} coincide with short exact sequences \eqref{eq:ses_of_comp1'}, \eqref{eq:ses_of_comp2'}. 

Similarly for $I=\{\ell\}.$ We obtain $R=(-\infty,\ell]$ and $L=(-\infty,\ell).$ Therefore $F_L=F_{(-\infty,\ell)}=F_{\ell-1},$ $F_{L-(r-1)} = F_{(-\infty,\ell-r+1)} = F_{\ell-r}.$ 
\end{proof}

\begin{proposition}\label{prop:SH_I=R}
Let $C$ be an $\RR$-filtered complex such that $C=\bigcup_\ell F_\ell C.$ Then for any $r\geq 0$ we have 
\begin{equation}
H_n(C) = \SH^r_{n,\RR}(C).
\end{equation}
\end{proposition}
\begin{proof}
In this case $R=R+r=\RR$ and $L=L-r=\emptyset.$ Therefore $F_R=F_{R+r}=C$ and $F_L=F_{L-r}=0.$ The assertion follows. 
\end{proof}

\begin{remark}[Naturalness in the interval and persistent spectral homology]\label{rem:nat_int}
Note that, if $I \preccurlyeq J,$ then $I^{+r} \preccurlyeq J^{+r}$ and $I_{-r} \preccurlyeq J_{-r}.$ Therefore, the commutative diagram 
\begin{equation}
\begin{tikzcd}
H_n(C_{I_{-r}})\ar[r]\ar[d] & H_n(C_{I^{+r}}) \ar[d] \\
 H_n(C_{J_{-r}}) \ar[r] & H_n(C_{J^{+r}})
\end{tikzcd}
\end{equation}
defines a morphism
\begin{equation}
f_{I,J}:\SH^r_{n,I}(C) \longrightarrow \SH^r_{n,J}(C).
\end{equation}
Moreover, it is easy to see that for $I\preccurlyeq J \preccurlyeq K$ we have $f_{J,K} f_{I,J}=f_{I,K}.$ In particular, it follows that we can consider persistent module ${\rm PSH}^r_n(C)$ such that 
\begin{equation}
{\rm PSH}^r_n(C)(\ell) = \SH^r_{n,(-\infty,\ell]}(C).
\end{equation}
\end{remark}

\begin{proposition}
For any $r,s\geq 0$ there is an isomorphism 
\begin{equation}
\SH^{r+s}_{n,I}(C)\cong \Im(\SH^r_{n,I_{-s}}(C) \longrightarrow \SH^r_{n,I^{+s}}(C)).
\end{equation}
Moreover, 
\begin{equation}
{\rm SB}^r_{n,I}(C) \subseteq  {\rm SB}^{r+s}_{n,I}(C), \hspace{1cm} {\rm SZ}^{r+s}_{n,I}(C)\subseteq {\rm SZ}^r_{n,I}(C),   
\end{equation}
and hence, $\SH^{r+s}_{n,I}(C)$ is a subquotient of $\SH^r_{n,I}(C).$
\end{proposition}
\begin{proof}
If $I=R\setminus L,$ then  $(I_{-s})_{-r}=I_{-(s+r)},$ $(I^{+s})^r=I^{+(s+r)},$ $(I^{+s})_{-r}=(R+s)\setminus (L-r)$ and $(I_{-s})^{+r}=(R+r)\setminus (L-s).$ Then $\SH^r_{n,I_{-s}}(C)$ and $\SH^r_{n,I^{+s}}(C)$ are images of the horizontal arrows in the following commutative diagram. 
\begin{equation}
\begin{tikzcd}
H_n(C_{I_{-(s+r)}}) \ar[r] \ar[d] & H_n(C_{(R+r)\setminus (L-s)})  \ar[d] \\
H_n(C_{(R+s)\setminus (L-r)}) \ar[r] & H_n(C_{I^{+s+r}})
 \end{tikzcd}       
\end{equation} 
Therefore the image of $\SH^r_{n,I_{-s}}(C)\to \SH^r_{n,I^{+s}}(C)$ is equal to the image of the diagonal map $H_n(C_{I_{-(s+r)}}) \to H_n(C_{I^{+(s+r)}}).$ 

The inclusion ${\rm SB}^r_{n,I}(C) \subseteq  {\rm SB}^{r+s}_{n,I}(C)$ follows from the fact that the map $C_{I} \to C_{I^{+(r+s)}}$ factors as $C_I \to C_{I^{+r}} \to C_{I^{+(r+s)}}.$ The second inclusion can be proved similarly. 
\end{proof}

\subsection{\texorpdfstring{$r$}{}-homotopy of \texorpdfstring{$\RR$}{}-filtered chain complexes} 

Similarly to \cite[Ch.XV,\S 3]{cartan1999homological} for $r\geq 0$ we say that two morphisms of $\RR$-filtered chain complexes  $f,f':C\to C'$ are $r$-homotopic, if there is a morphism of graded modules of degree one  $h:C_*\to C'_{*+1}$ such that $\partial h + h \partial = f-f'$ and $h(F_\ell C) \subseteq F_{\ell+r}C'.$

\begin{proposition}[{cf. \cite[Ch.XV,Prop.3.1]{cartan1999homological}}]\label{prop:r-homotopic-R-filtered}
Any $r$-homotopic morphisms of $\RR$-filtered chain complexes $f,f':C\to C'$ induce equal morphisms on the $r$-spectral homology
\begin{equation}
\SH_{n,I}^r(f)=\SH_{n,I}^r(f') : \SH_{n,I}^r(C) \longrightarrow \SH_{n,I}^r(C').
\end{equation}
\end{proposition}
\begin{proof}
We present $I$ as the difference of left infinite intervals $I=R\setminus L.$ We set $F_X=F_XC$ and $F'_X=F_XC'.$  Then
\begin{equation}
\begin{split}
\SH^r_{n,I}(C) &\cong \Im( \alpha_* : H_n(F_R/F_{L-r}) \longrightarrow H_n(F_{R+r}/F_L)),\\
\SH^r_{n,I}(C') &\cong \Im( \alpha'_* : H_n(F'_R/F'_{L-r}) \longrightarrow H_n(F'_{R+r}/F'_L)).
\end{split}
\end{equation} 
Here we denote by $\alpha$ and $\alpha'$ the morphisms of chain complexes induced by embeddings
\begin{equation}
\alpha:F_R/F_{L-r}\to F_{R+r}/F_L, \hspace{1cm}    \alpha':F'_R/F'_{L-r} \to F'_{R+r}/F'_L.
\end{equation}
Set $g=f-f'$ and consider the commutative diagram 
\begin{equation}
\begin{tikzcd}
H_n(F_R/F_{L-r})\ar[r,"\alpha_*"]\ar[d,"g_*"] & H_n(F_{R+r}/F_L) \ar[d,"g_*"] \\ 
H_n(F'_R/F'_{L-r})\ar[r,"\alpha'_*"] & H_n(F'_{R+r}/F'_L).
\end{tikzcd}
\end{equation}
Then it is sufficient to prove that $g_*\alpha=\alpha'g^*=0.$ Take an element $x\in (F_R)_n$ that represents an element from $H_n(F_R/F_{L-r}).$ Then $\partial (x)\in (F_{L-r})_{n-1}.$ Since $g=\partial h + h \partial,$ 
we have
\begin{equation}
g(x) =  \partial h(x) +  h \partial (x).
\end{equation}
Using that $h$ is an $r$-homotopy and $\partial(x)\in (F_{L-r})_{n-1},$ we obtain $h\partial(x)\in (F'_L)_n.$ Hence  $h\partial(x)$ represents a zero element in $H_n(F'_{R+r}/F'_L).$ Using that $h$ is an $r$-homotopy once again, we obtain $h(x)\in (F'_{R+r})_{n+1}.$ Therefore $\partial h(x)$ represents a zero element in $H_n(F'_{R+r}/F'_L)$ too. It follows that $g(x)$ represents zero element in $H_n(F'_{R+r}/F'_L).$ So we obtain  $g_*\alpha = \alpha'_*g_*=0.$
\end{proof}

\subsection{Spectral homology of a quasimetric space}

Let $\KK$ be a commutative ring and $X$ be a quasimetric space. Define the $\RR$-filtered reachability chain complex ${\rm RC}(X)$ similarly to the case of digraphs. Its $n$-th component is freely generated by tuples  $(x_0,\dots,x_n)$ without consecutive repetitions $x_i\neq x_{i+1}$ such that $d(x_i,x_{i+1})<+\infty.$ 
\begin{equation}
{\rm RC}_n(X) =\KK\cdot \{(x_0,\dots,x_n)\mid d(x_i,x_{i+1})<+\infty, x_i\ne x_{i+1}\}.
\end{equation}
In this complex we identify sequences with consecutive repetitions with zero. Then the differential is defined by the formula
\begin{equation}
\partial(x_0,\dots,x_n) = \sum_{i=0}^n(-1)^i (x_0,\dots,\hat x_i,\dots,x_n).
\end{equation}
If we set 
$L(x_0,\dots,x_n)=\sum_i d(x_i,x_{i+1}),$
the $\RR$-filtration is defined so that for a real number $r$  the component $F_r{\rm RC}_n(G)$ is generated by tuples $(x_0,\dots,x_n)$ such that $L(x_0,\dots,x_n)\leq r.$

For a non-empty interval $I,$ a real number $r\geq 0,$ and natural $n$ the spectral homology of $X$ is defined as the spectral homology of this $\RR$-filtered chain complex
\begin{equation}
\SH^r_{n,I}(X) = \SH^r_{n,I}( {\rm RC}(X)). 
\end{equation}
It is easy to see that these constructions are natural and they define functors
\begin{equation}
\SH^r_{n,I} : {\sf QMet} \longrightarrow {\sf Mod}(\KK).
\end{equation}
By Remark \ref{rem:nat_int} the spectral homology is natural by the interval. In particular, we can consider the persistent spectral homology ${\rm PSH}^r_n(X)$ which is a persistent module  such that ${\rm PSH}^r_n(X)=\SH^r_{n,[0,\ell]}(X).$

\begin{theorem}
Any $r$-homotopic morphisms of quasimetric spaces  $\varphi,\psi:X\to Y$ induce equal morphisms on $r$-spectral homology
\begin{equation}
\SH^r_{n,I}(\varphi)=\SH^r_{n,I}(\psi) : \SH^r_{n,I}(X) \longrightarrow \SH^r_{n,I}(Y)
\end{equation}
 for any $n$ and $I$.
\end{theorem}
\begin{proof}  The differential in the reachability complex is defined by the formula $\partial_n=\sum_i (-1)^i \partial_{n,i},$ where
\begin{equation}
\partial_{n,i}(x_0,\dots,x_n) = (x_0,\dots,\hat x_i,\dots,x_n).
\end{equation}
Set $f=\RC(\varphi)$ and $f'=\RC(\psi).$ For $0\leq j\leq n$ we
consider the map $h_{n,j}:\RC_n(X)\to \RC_{n+1}(Y)$ given by 
\begin{equation}
h_{n,j}(x_0,\dots,x_n) = (\varphi(x_0),\dots,\varphi(x_j),\psi(x_j), \dots,\psi(x_n))).
\end{equation}
These maps satisfy simplicial homotopy relations 
\begin{equation}
\partial_{n+1,0}h_{n,0} = f'_n, \hspace{1cm} \partial_{n+1,n+1}h_{n,n}=f_n, 
\end{equation}
\begin{equation}
\partial_{n+1,i}h_{n,j} = \begin{cases}
h_{n-1,j-1}\partial_{n,i}, & i<j,\\
\partial_{n,i}h_{n,i-1}, & i=j\ne 0,\\
h_{n-1,j} \partial_{n,i-1}, & i>j+1.
\end{cases}
\end{equation}
Therefore, if we set $h_n=\sum_j (-1)^j h_{n,j},$ we obtain $\partial_{n+1}h_n+h_{n-1}\partial_n = f'_n-f_n.$

If $d(\varphi,\psi)\leq r,$ then
\begin{equation}
L(h_{n,j}(x_0,\dots,x_n)) \leq L(x_0,\dots,x_n)+r.  
\end{equation}
It follows that $h:\RC_*(X)\to \RC_{*+1}(Y)$ is an $r$-homotopy between morphisms $\RR$-filtered chain complexes. Hence the statement follows from Proposition \ref{prop:r-homotopic-R-filtered}.
\end{proof}
\begin{corollary}
If $X$ and $Y$ are $r$-homotopy equivalent quasimetric spaces, then 
\begin{equation}
\SH^r_{n,I}(X)\cong \SH^r_{n,I}(Y).
\end{equation}
In particular, if $X$ is finite, then
$\SH^r_{n,I}(X)\cong \SH^r_{n,I}(M_r(X)).$
\end{corollary}

\subsection{Relation of spectral homology with other types of homology}

Let $X$ be a quasimetric space. Then in our notation its magnitude homology can be defined as
\begin{equation}\label{eq:MH}
\MH_{n,\ell}(X) = H_n\left(\frac{F_\ell \RC(X)}{F_{(-\infty ,\ell)}\RC(X)}\right),
\end{equation}
and the 
blurred magnitude homology \cite{govc2021persistent} are defined as a persistent module such that 
\begin{equation}
{\rm BMH}_{n}(X)(\ell) = H_n(F_\ell \RC(X)).
\end{equation}

\begin{theorem}\label{th:other_types_of_homology} Let $X$ be a quasimetric space and $G$ be a digraph. Then the following holds.
\begin{enumerate}
\item The magnitude homology can be presented as $0$-spectral homology
\[
{\rm MH}_{n,\ell}(X)\cong \SH^0_{n,\{\ell\}}(X).\]
In particular, for any $r\geq 0$ the group $\SH^{r}_{n,\{\ell\}}(X)$ is an $r$-homotopy invariant subquotient of ${\rm MH}_{n,\ell}(X).$
\item The persistent module of  blurred magnitude homology (see \cite{govc2021persistent}) can be presented as persistent $0$-spectral homology
\[
 {\rm BMH}_{n}(X)\cong {\rm PSH}^0_{n}(X).   
\]
In particular, for any $r\geq 0$ the persistent module ${\rm PSH}^r_{n}(X)$ is an $r$-homotopy invariant subquotient persistent  module of ${\rm BHM}_n(X).$ 
\item The reachability homology (see \cite{hepworth2023reachability}) can be presented as $r$-spectral homology  
\[{\rm RH}_n(X)\cong \SH^r_{n,\RR}(X)\] for any $r\geq 0.$
\item For any integer $r\geq 1$ the $r$-th page of the magnitude-path spectral sequence of $G$ can be presented as $(r-1)$-spectral homology 
\[E^r_{n,\ell-n}(G)\cong \SH^{r-1}_{n,\{\ell\}}(G).\]
\item In particular, path homology can be described as
\[
{\rm PH}_n(G)=\SH^1_{n,\{n\}}(G).\]
\end{enumerate}
\end{theorem}
\begin{proof}
(1) and (2) obviously follow from the definitions equation  \eqref{eq:SH^0} and Remark \ref{rem:nat_int}. (3) follows from Proposition  \ref{prop:SH_I=R}. (4) follows from Proposition \ref{prop:spectral=spectral}.
\end{proof}

\subsection{Low-dimensional spectral homology}

In this section we fix a quasimetric space $X$ and we will omit it in some notations. For an interval $I$ we denote by $\sim_I$ the equivalence relation on $X$ such that $x\sim_I y$ if and only if there is $s\in I$ such that $x\sim_s y.$

\begin{proposition} Let $I$ be an interval and $r\geq 0$ be a real number. Then, assuming $0\in I,$ we have
\begin{equation}
\SH^r_{0,I}(X) \cong  \KK\cdot (X/\sim_{I^{+r}}). 
\end{equation}
In particular, 
\begin{equation}
\SH^r_{0,\{0\}}(X) \cong  \KK\cdot (X/\sim_r).
\end{equation}
If $0\notin I,$ then $\SH^r_{0,I}(X)=0.$
\end{proposition}
\begin{proof}  We set $C_I=F_{R_I}\RC(X)/F_{L_I}\RC(X).$ If $0\notin I,$ then $(C_I)_0=0$ and $\SH^r_{0,I}(X)=0.$ If $0\in I,$ we have $(C_I)_0=\KK\cdot X$ and $(C_I)_1 = \KK\cdot \{(x,y)\mid d(x,y)\in I\}.$ Therefore $\Im(\partial_1:(C_I)_1\to (C_I)_0)$ is spanned by differences $x-y,$ where $d(x,y)\in I.$ It follows that $\SH^0_{0,I}(X)=\KK\cdot(X/\sim_I).$ Hence $\SH^r_{0,I}(X)=\Im(\KK\cdot(X/\sim_{I_{-r}}) \to \KK\cdot(X/\sim_{I^{+r}}))=\KK\cdot(X/\sim_{I^{+r}}).$
\end{proof}

A pair of points $(x,y)$ of a quasimetric space $X$ is called $\ell$-adjacent, if $d(x,y)=\ell$ and there is no a point $a$ such that $d(x,a)+d(a,y)=\ell$ and $a\ne x,y.$ The set of $\ell$-adjacent pairs is denoted by ${\rm Adj}_\ell.$ Then it is known  \cite[Cor. 4.5]{leinster2021magnitude}  that 
\begin{equation}
\MH_{1,\ell}(X) \cong  \KK\cdot {\rm Adj}_\ell.
\end{equation}
We will prove a generalization of this statement describing $\SH^r_{1,\{\ell\}}(X)$ for $\ell>r\geq 0.$

Let $\ell>r\geq 0.$ Denote by $P_{\leq \ell+r}$ the set of all pairs of points $(x,y)$ such that $d(x,y)\leq \ell+r.$ 
\begin{equation}
P_{\leq \ell+r} = \{ (x,y)\in X^2\mid d(x,y)\leq \ell+r \}.
\end{equation}
We consider the minimal equivalence relation $\sim$ on $P_{\leq \ell+r}$ such that the following holds
\begin{itemize}
    \item $(x,y)\sim (x',y),$ if $d(x,x')\leq r$ and  $d(x,x')+d(x',y)\leq \ell+r,$
    \item  $(x,y)\sim (x,y'),$ if $d(y',y)\leq r$ and  $d(x,y')+d(y',y)\leq \ell+r.$
\end{itemize}

We say that a pair $(x,y)\in P_{\leq \ell+r}$ is $(\ell,r)$-trivial, if there exists a point $a$ such that 
\begin{equation}
d(x,a)+d(a,y)\leq \ell+r, \hspace{5mm} d(x,a)<\ell, \hspace{5mm} d(a,y)<\ell. 
\end{equation}
In particular, if $d(x,y)<\ell,$ then $(x,y)$ is $(\ell,r)$-trivial. We denote the set of $(\ell,r)$-trivial pairs by 
\begin{equation}
T_{\ell,r}\subseteq P_{\leq \ell+r}.
\end{equation}

We say that a pair $(x,y)$ is homotopy $(\ell,r)$-trivial, if it is equivalent to a trivial pair with respect to the equivalence relation defined above. The set of homotopy $(\ell,r)$-trivial pairs is denoted by $\overline{T}_{\ell,r}$
\begin{equation}
    T_{\ell,r} \subseteq \overline{T}_{\ell,r} \subseteq P_{\leq \ell+r}.
\end{equation}
A pair $(x,y)\in P_{\leq \ell+r}$ is called $(\ell,r)$-adjacent, if it is not homotopy $(\ell,r)$-trivial
\begin{equation}
{\rm Adj}_{\ell,r} = P_{\leq \ell+r}\setminus \overline{T}_{\ell,r}. 
\end{equation}
We consider the restricted equivalence relation on ${\rm Adj}_{\ell,r}.$

\begin{proposition}\label{prop:sh^0_{1,l,l+r}}
Let $X$ be a quasimetric space and $\ell>r\geq 0.$ Then 
\begin{equation}
\SH^0_{1,[\ell,\ell+r]}(X) \cong \KK\cdot ({\rm Adj}_{\ell,r}/\sim).
\end{equation}
\end{proposition}
\begin{proof}
Set $C=F_{[\ell,\ell+r]}\RC(X).$ Then $C_0=0,$ 
\begin{equation}
 C_1=\KK\cdot \{(x,y)\mid \ell \leq  d(x,y)\leq \ell+r\},
 \end{equation}
 \begin{equation}
 C_2=\KK\cdot \{(x,a,y)\mid \ell\leq d(x,a)+d(a,y)\leq \ell+r\}.   
\end{equation}
Consider the map $\theta: C_1\to \KK\cdot ({\rm Adj}_{\ell,r}/\sim)$ defined so that it sends pairs from $\overline{T}_{\ell,r}$ to zero, and all other pairs $(x,y)\in {\rm Adj}_{\ell,r}$ to their equivalence classes $\theta(x,y)=[(x,y)].$ It is easy to see that $\theta$ is surjective. We need to prove that 
\begin{equation}
\Ker(\theta) = \Im(\partial:C_2\to C_1).
\end{equation}

Let us prove that  $\Im(\partial : C_2\to C_1)\subseteq \Ker(\theta).$ If $(x,a,y)\in C_2,$ then either $d(x,a)<\ell$ or $d(a,y)<\ell,$ because if $d(x,a)\geq \ell$ and $d(a,y)\geq \ell,$ then $d(x,a)+d(a,y)\geq 2\ell>\ell+r.$ Assume first that $d(x,a)<\ell.$ Then we have four cases.

Case 1.1: $d(a,y)<\ell$ and $d(x,y)<\ell.$ Then $\partial(x,a,y)=0.$

Case 1.2: $d(a,y)<\ell$ and $d(x,y)\geq \ell.$ Then $\partial(x,a,y)=-(x,y).$ In this case $(x,y)\in T_{\ell,r}$ and hence $\theta(\partial(x,a,y))=0.$ 

Case 1.3: $d(a,y)\geq \ell$ and $d(x,y)<\ell.$ Then $d(x,a)\leq r$ and  $\partial(x,a,y)=(a,y).$ In this case  $(x,y)\sim (a,y)$ and $(x,y)\in T_{\ell,r}.$ In particular, $ (a,y)\in \overline{T}_{\ell,r}$ and $\theta(\partial(x,a,y))=0.$

Case 1.4: $d(a,y)\geq \ell$ and $d(x,y)\geq \ell.$ Then $d(x,a)\leq r$ and  $\partial(x,a,y)= -(x,y)+(a,y).$ In this case $(x,y)\sim (a,y),$ and $\theta(\partial(x,a,y))=0.$ 

If $d(a,y)<\ell,$ we have symmetric cases. So we proved that $\Im(\partial:C_2\to C_1)\subseteq \Ker(\theta).$ 

Let us prove that $\Ker(\theta) \subseteq \Im(\partial:C_2\to C_1).$ First we need to check that if $(x,y),(x',y')\in P_{\leq \ell+r}$ and  $(x,y)\sim (x',y'),$ then $(x,y)-(x',y')\in \Im(\partial:C_2\to C_1).$ Here we assume that a pair $(x,y)\in P_{\leq \ell+r}$ such that $d(x,y)<\ell$ represents the zero element in $C_1,$ and a triple $(x,a,y)$ such that $d(x,a)+d(a,y)<\ell$ represents the zero element in $C_2.$ 

Let $(x,y),(x',y)\in P_{\leq \ell+r}$ such that  $d(x,x')\leq r$ and  $d(x,x')+d(x',y)\leq \ell+r.$ Then $(x,x',y)\in C_2$ and $\partial(x,x',y)=-(x,y)+(x',y).$ Therefore $(x,y)-(x',y)\in \Im(\partial).$ The case of pairs $(x,y),(x,y')\in P_{\leq \ell+r}$ such that $d(y',y)\leq r$ and $d(x,y')+d(y',y)\leq \ell+ r$ is similar.  It follows that the differences of equivalent pairs from $P_{\leq \ell+r}$ are in the image of $\partial.$

Now assume that $(x,y)\in T_{\ell,r}.$ Then there exists $a$ such that $d(x,a)+d(a,y)\leq \ell+r$ and $d(x,a),d(a,y)<\ell.$ Therefore $(x,a,y)\in C_2$ and $\partial(x,a,y)=-(x,y).$ It follows that $\Ker(\theta) \subseteq \Im(\partial:C_2\to C_1).$
\end{proof}

\begin{theorem}
Let $X$ be a quasimetric space and $\ell>r\geq 0.$ Then there is an isomorphism 
\begin{equation}
\SH^r_{1,\{\ell\}}(X)\cong \KK\cdot (({\rm Adj}_\ell\cap {\rm Adj}_{\ell,r})/\sim), 
\end{equation}
where $\sim$ is the equivalence relation induced from the equivalence relation on ${\rm Adj}_{\ell,r}.$
\end{theorem}
\begin{proof} We set $C_I=F_R\RC(X)/F_L\RC(X).$ 
Since $r<\ell$ we have $(C_{[\ell-r,\ell]})_0=0.$ Then we obtain that any element in $(C_{[\ell-r,\ell]})_1$  is a cycle. It follows that 
\begin{equation}
\SH^r_{1,\{\ell\}}(X) = \Im((C_{[\ell-r,\ell]})_1 \longrightarrow \SH^0_{1,[\ell,\ell+r]}(X)).
\end{equation}
Therefore, $\SH^r_{1,\{\ell\}}(X)$ is a submodule of $\SH^0_{1,[\ell,\ell+r]}(X)$ generated by pairs $(x,y)$ such that $d(x,y)=\ell.$ Then the statement follows from Proposition \ref{prop:sh^0_{1,l,l+r}}.
\end{proof}

Assume that $x,y\in \RR^n.$  Then we consider the following filled ellipsoid of revolution with foci $x,y$ 
\begin{equation}
{\rm Ell}_{r}(x,y) = \{a\in \RR^n \mid d(x,a)+d(a,y)\leq d(x,y)+r\}.
\end{equation}
We will also consider open balls of radius $\ell$
\begin{equation}
B^\circ_\ell(x) = \{a\in \RR^n\mid d(x,a)<\ell\}.
\end{equation}
We define the \emph{$r$-thick open interval} between $x,y$ as
\begin{equation}
I_{-r}(x,y) = B^\circ_{d(x,y)}(x) \cap {\rm Ell}_{r}(x,y) \cap B^\circ_{d(x,y)}(y).
\end{equation}
It can be depicted as the following gray area between $x$ and $y$.
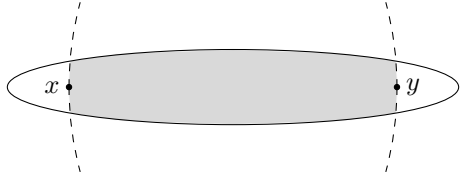
\[
\begin{tikzpicture}
\draw [color=white, fill=gray!30] plot  coordinates {
(-2.19,0) 
(-2.17,0.35)
(-1.5,0.45)
(-1,0.47) 
(0,0.5) 
(1,0.47) 
(1.5,0.45)
(2.17,0.35) 
(2.19,0) 
(2.17,-0.35)
(1.5,-0.45)
(1,0.-0.47) 
(0,-0.5) 
(-1,-0.47) 
(-1.5,-0.45)
(-2.17,-0.35) 
(-2.19,0)
};
\draw (0,0) ellipse (3cm and 5mm);
    \filldraw  (2.18,0) circle (1pt) node[anchor=west]{$y$};
    \filldraw  (-2.18,0) circle (1pt) node[anchor=east]{$x$};
    \draw[dashed] (-2.18,0) arc (180:165:4.36);
    \draw[dashed] (-2.18,0) arc (0:15:-4.36);
    \draw[dashed] (2.18,0) arc (0:15:4.36);
    \draw[dashed] (2.18,0) arc (180:165:-4.36);
\end{tikzpicture}
\]

So for $x,y\in X$ such that $d(x,y)=\ell$ we have $(x,y)\in T_{\ell,r}$ if and only if $I_{-r}(x,y)\cap X\ne \emptyset.$

\begin{corollary}
Let $\ell>r\geq 0$ and $X\subseteq \RR^n.$ Assume that for each $x,y\in X$ such that $d(x,y)=\ell,$ the intersection of $I_{-r}(x,y)$ with $X$ is not empty. Then 
\begin{equation}
\SH^r_{1,\{\ell\}}(X)=0.    
\end{equation}
\end{corollary}
\begin{proof} In this case all elements from ${\rm Adj}_\ell$ are in $T_{\ell,r},$ and ${\rm Adj}_\ell \cap {\rm Adj}_{\ell,r}=\emptyset.$
\end{proof}

Assume that $X\subseteq \RR^n.$ We say that a (continuous) path $\gamma:[0,1]\to X$ from $x=\gamma(0)$ reduces the distance to $y\in X,$  if  $d(x,y)>d(\gamma(t),y)$ for any $0<t\leq 1$ (we do not assume that $\gamma(1)=y$). In other words, $\gamma$ reduces the distance to $y,$ if 
\begin{equation}
\gamma((0,1])\subseteq B^\circ_{d(\gamma(0),y)}(y).
\end{equation}
For example, the path from $x$
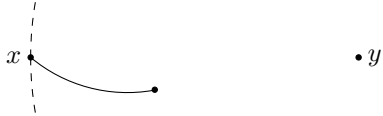
\[
\begin{tikzpicture}
\draw (-2.18,0) arc (230:280:2);
\filldraw  (-2.18,0) circle (1pt) node[anchor=east]{$x$};
\filldraw  (2.18,0) circle (1pt) node[anchor=west]{$y$};
\filldraw  (-0.53,-0.43) circle (1pt);
   \draw[dashed] (-2.18,0) arc (180:170:4.36);
    \draw[dashed] (-2.18,0) arc (0:10:-4.36);
\end{tikzpicture}
\]
reduces the distance to $y.$ Note that, if $r>0,$ and $\gamma$ is a path from $x$ reducing the distance to $y$ we have $\gamma(\varepsilon)\in I_{-r}(x,y)$ for  small enough $\varepsilon>0.$ 

\begin{corollary}\label{cor:path}
Let $\ell>r> 0$ and $X\subseteq \RR^n.$ If $x,y\in X$ such that $d(x,y)=\ell$ and there is a path from $x$ reducing the distance to $y,$ then $(x,y),(y,x)\in T_{\ell,r}.$ In particular, if for any distinct $x,y\in X$ there is either a path from $x$ reducing the distance to $y,$ or a path from $y$ reducing the distance to $x,$ then $\SH^r_{1,\{\ell\}}(X)=0.$ 
\end{corollary}

\begin{example}
Consider the standard circle $S^1\subseteq \RR^2.$ By Corollary \ref{cor:path} for any $0<r<\ell$ we have $\SH^r_{1,\{\ell\}}(S^1)=0.$ 
\end{example}

\begin{example}
Let $X\subseteq S^1$ be an arc lying in one open semicircle. 
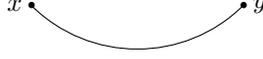
\[
\begin{tikzpicture}
\draw (0,0) arc (-45:-135:2);
\filldraw  (0,0) circle (1pt) node[anchor=west]{$y$};
\filldraw  (-2.82,0) circle (1pt) node[anchor=east]{$x$};
\end{tikzpicture}
\]
In other words the anticlockwise directed distance from $x$ to $y$ is strictly less than $\pi.$  
Then by Corollary \ref{cor:path} we obtain $\SH^r_{1,\{\ell\}}(X)=0$ for any $0<r<\ell.$
\end{example}

\begin{example}
Now assume that $X\subseteq S^1$ is an arc that does not lie in one semicircle. 
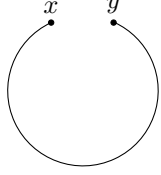
\[
\begin{tikzpicture}
\draw (0,0) arc (115:425:1);
\filldraw  (0,0) circle (1pt) node[anchor=south]{$x$};
\filldraw  (0.83,0) circle (1pt)
node[anchor=south]{$y$};
\end{tikzpicture}
\]
So the anticlockwise directed distance from $x$ to $y$ is strictly greater than $\pi.$ Moreover, let us assume that the points are close enough $0<d(x,y)<\sqrt{2}/2.$ Then we claim that for $0<r<\ell$
\begin{equation}
\SH^r_{1,\{\ell\}}(X) =
\begin{cases}
 \ZZ\cdot \{(x,y),(y,x)\},& \ell=d(x,y),   \\
 0, & \ell\ne d(x,y)
\end{cases}
\end{equation} 
(in this case $\SH^r_{1,\{\ell\}}(X)$ has finite rank, while $\SH^0_{1,\{\ell\}}(X)=\MH_{1,\ell}(X)$ has  infinite rank).
Let us prove this. 

If $\ell\ne d(x,y),$ then for any  $x',y'\in X$ such that $d(x',y')=\ell$ there is either a path from $x$ reducing the distance to $y,$ or a path from $y$ reducing the distance to $x.$ So $\SH^r_{1,\ell}(X)=0.$

Now assume that $\ell=d(x,y).$ Then all pairs $(x',y')\in {\rm Adj}_\ell$ except $(x,y)$ and $(y,x)$ lie in $T_{\ell,r}$ because there is a path from $x'$ reducing the distance to $y'.$ We claim that $(x,y),(y,x) \notin \overline{T}_{\ell,r}.$ For this we write everything in coordinates in a symmetric form $y=(a,b)$ and $x=(-a,b).$ If $(x,y)\sim (x',y')\in P_{\leq \ell+r},$ using that $\ell+r< \sqrt{2}$ and $r<\ell,$ one can show that $x'$ is in the Quadrant two and $y'$ is in the Quadrant one.  
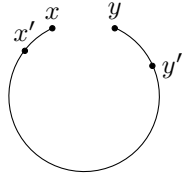
\[
\begin{tikzpicture}
\draw (0,0) arc (115:425:1);
\filldraw  (0,0) circle (1pt) node[anchor=south]{$x$};
\filldraw  (0.83,0) circle (1pt)
node[anchor=south]{$y$};
\filldraw  (-0.365,-0.3) circle (1pt) node[anchor=south]{$x'$};
\filldraw  (1.33,-0.5) circle (1pt) node[anchor=west]{$y'$};
\end{tikzpicture}
\]
Then there is no a point $a\in X$ such that $d(x',a)<\ell$ and $d(a,y')<\ell.$  Hence $(x,y)\notin \overline{T}_{\ell,r}.$ This also implies that $(x,y)$ is not equivalent to $(y,x).$ Therefore $(({\rm Adj}_\ell\cap {\rm Adj}_{\ell,r})/\sim) \cong  \{(x,y),(y,x)\}.$
\end{example}

\printbibliography

\end{document}